\documentclass[10pt]{article}
\usepackage{amsmath,amssymb,amsthm,graphicx,graphics}
\usepackage{amscd,t1enc}
\usepackage[latin2]{inputenc}
\usepackage[usenames,dvipsnames]{color}

\definecolor{xfqqff}{rgb}{0.4980392156862745,0.,1.}
\definecolor{zzffff}{rgb}{0.6,1.,1.}
\definecolor{ffzzcc}{rgb}{1.,0.6,0.8}
\definecolor{ffdxqq}{rgb}{1.,0.8431372549019608,0.}
\definecolor{qqwuqq}{rgb}{0.,0.39215686274509803,0.}
\definecolor{sqsqsq}{rgb}{0.12549019607843137,0.12549019607843137,0.12549019607843137}
\definecolor{yqyqyq}{rgb}{0.5019607843137255,0.5019607843137255,0.5019607843137255}
\definecolor{qqffqq}{rgb}{0.,1.,0.}
\definecolor{ffqqqq}{rgb}{1.,0.,0.}
\definecolor{ttttff}{rgb}{0.2,0.2,1.}

\newcommand{\cut}[1]{}

\newtheorem{thm}{Theorem}[section]
\newtheorem{theorem}[thm]{Theorem}

\newtheorem{proposition}[thm]{Proposition}
\newtheorem{lemma}[thm]{Lemma}
\newtheorem{remark}[thm]{Remark}

\newtheorem{definition}[thm]{Definition}

\newtheorem{result}[thm]{Result}

\newcommand{\cC}{\mathcal{C}}

\newcommand{\cH}{\mathcal{H}}

\newcommand{\PG}{\mathrm{PG}}

\newcommand{\AG}{\mathrm{AG}}
\newcommand{\GF}{\mathrm{GF}}
\newcommand{\UCN}{\bar\chi}

\newcommand{\qbinom}[2]%
           {\left[\!\begin{smallmatrix}#1\\#2\end{smallmatrix}\!\right]_q}

\newcommand{\PP}{{\mathbb{P}}}

\newcommand{\ob}{\overline{\chi}_b}

\usepackage{graphicx}
\usepackage{subfig}
\usepackage{array,booktabs}
\newcolumntype{M}[1]{>{\centering\arraybackslash}m{#1}}
\usepackage{multirow}

\begin{document}

\title {On the balanced upper chromatic number of finite
        projective planes}
 \vspace{3mm}
\author{\begin{tabular}{c}
Zolt\'an L. Bl\'azsik\footnote{MTA--ELTE Geometric and Algebraic Combinatorics Research Group, 1117 Budapest, P\'azm\'any P.\ stny.\ 1/C, Hungary} \qquad\qquad Aart Blokhuis\footnote{Eindhoven University of Technology, Eindhoven, The Netherlands} \qquad\qquad \v{S}tefko Miklavi\v{c}\footnote{Andrej Maru\v{s}i\v{c} Institute, University of Primorska, Koper, Slovenia} \\ ~\\
Zolt\'an L\'or\'ant Nagy\footnotemark[1] \qquad\qquad Tam\'as Sz\H{o}nyi\footnotemark[1]\textsuperscript{~,}\footnote{ELTE E\"otv\"os Lor\'and University, Budapest, Hungary}
\end{tabular}}

\maketitle

\cut{
\begin{abstract}
In this paper, we determine the
balanced upper chromatic number of the desarguesian
projective plane $\PG(2,q)$ for all $q$.
Moreover, we give colorings of projective planes represented by cyclic,
affine and relative difference sets (or planar functions) giving
a lower bound on the balanced upper chromatic number of the right
order of magnitude.
Furthermore, if $q\not\equiv 0 \pmod 3$ then we show a balanced coloring
with 3 and 4 element color classes of an affine plane of order $q$ represented
by a planar function.
We finally give a general lower bound using a probabilistic argument.
\end{abstract}
}

\begin{abstract}
In this paper, we study vertex colorings of hypergraphs in which all color class sizes differ by at most  one (balanced colorings) and each hyperedge contains at least two vertices of the same color (rainbow-free colorings). For any  hypergraph $H$, the maximum number $k$ for which  there is a balanced rainbow-free $k$-coloring of $H$ is called the balanced upper chromatic number of the hypergraph. We confirm the conjecture of Araujo-Pardo, Kiss and Montejano by determining the balanced upper chromatic number of the desarguesian projective plane $\PG(2,q)$ for all $q$. In addition, we determine asymptotically the balanced upper chromatic number of several families of non-desarguesian projective planes and also provide a general lower bound for arbitrary projective planes using probabilistic methods which determines the parameter up to a multiplicative constant.
\end{abstract}

\section{Introduction}

In recent years the notion of a proper strict coloring of hypergraphs
was investigated in several papers by Voloshin, Bacs\'o, Tuza and others,
including \cite{Iljics}, \cite{V2}, \cite{TB} and \cite{BHSz}. In this work,
instead of studying the upper chromatic number we will focus on improving the
known estimates of the balanced upper chromatic number of such hypergraphs
which arise from projective planes.

Let $\cH$ denote a hypergraph with vertex set $V$ ($|V|=v$) and
(hyper)edge set $E$.
A strict $N$-coloring $\cC$ of $\cH$ is a coloring of the vertices using
exactly $N$ colors; in other words, the collection $\cC=\{C_1,\ldots,C_N\}$ of color classes is a partition
of $V$. Given a coloring $\cC$, we
define the mapping $\varphi_\cC\colon V\to \{1,2,\ldots,N\}$ by
$\varphi_\cC(P)=i$ if and only if $P\in C_i$.
We call the numbers $1,\ldots, N$
colors and the sets $C_1,\ldots, C_N$ color classes. We call an edge
$H\in E$ \emph{rainbow (with respect to $\cC$)} if no two points of $H$ have
the same color; that is, $|H\cap C_i|\leq 1$ for all $1\leq i\leq N$.
The upper chromatic number of the hypergraph $\cH$, denoted by $\UCN(\cH)$,
is the maximal number $N$ for which $\cH$ admits a strict $N$-coloring
without rainbow edges.
Let us call such a coloring \emph{proper} or \emph{rainbow-free}.
A \emph{balanced coloring} is a coloring in which the
cardinality of any two color classes differs by at most one.
The \emph{balanced upper chromatic number}
of a hypergraph $\cH$, denoted by $\ob(\cH)$, is the largest integer $N$ such
that $\cH$ admits a proper strict balanced $N$-coloring.

In the following sections we will focus on hypergraphs
which arise from a projective plane
$\Pi$ (of order $q$).
The vertices are the points of the plane and the edges correspond
to the lines of the plane.
In 2015, Araujo-Pardo, Kiss and Montejano proved the following results.

\begin{result}[\cite{KissGy}] \label{AKMupper}
All balanced rainbow-free colorings of any projective plane of order $q$
satisfy that each color class contains at least three points. Thus
\[
    \ob(\Pi_q) \le \frac{q^2+q+1}{3}.
\]
\end{result}

\begin{result}[\cite{KissGy}] \label{AKMlower}
For every cyclic projective plane $\Pi_q$ we have
\[
    \ob(\Pi_q) \ge \frac{q^2+q+1}{6}.
\]
If the difference set defining $\Pi_q$ in  $\mathbb{Z}_{q^2+q+1}$
contains $\{0,1,3\}$ then
 \[
    \ob(\Pi_q) = \left \lfloor \frac{q^2+q+1}{3} \right \rfloor.
\]
\end{result}

We will use the last observation to determine the balanced upper
chromatic number of the desarguesian projective plane $\PG(2,q)$ in the second section.

In the third section, we will use some well-known representations (such as
affine and relative difference sets, planar functions) of projective planes of order $q$ (including non-desarguesians) in order to present general lower bounds on the balanced upper chromatic number. We managed to reach the correct order of magnitude for the remaining two cases, too. Moreover, we prove a sharp result if $q\equiv 0 \pmod{3}$.

\begin{theorem}\label{aff}
For $q\equiv 2 \pmod{3}$, let $\Pi_q$ be a projective plane of order $q$
represented by an affine difference set. Then
\[
    \ob(\Pi_q) \ge \frac{q^2+2}{3}.
\]
\end{theorem}

\begin{theorem}\label{par}
For $q\equiv 0 \pmod{3}$, let $\Pi_q$ be a projective plane of order $q$
represented by a planar function (or relative difference set). Then
\[
    \ob(\Pi_q) = \left \lfloor \frac{q^2+q+1}{3} \right \rfloor =
\frac{q^2+q}{3}.
\]
\end{theorem}

If $q\not\equiv 0 \pmod 3$, we manage to give a coloring of the affine plane of order $q$
represented by a suitable planar function, thus we get a lower bound on the balanced upper
chromatic number. Furthermore, as a consequence we have a coloring of the corresponding
projective plane which means that we have a lower bound on the balanced upper chromatic
number of any projective plane of order $q$ represented by a planar function.

\begin{theorem}\label{planar}
For $q\not\equiv 0 \pmod{3}$ and $p>5$, let $A_q$ and $\Pi_q$ be the affine and projective plane of order $q=p^h$ represented by a planar function, respectively. Then
\[
    \ob(A_q) \ge \left \{ \begin{array}{ll}
    \frac{\left ( 1 - \frac{1}{p}  \right )q^2}{3} & \mathrm{if~} p=3k+1, \\
    \frac{\left ( 1 - \frac{2}{p}  \right )q^2}{3} & \mathrm{if~} p=3k+2;
    \end{array} \right .
\]
\[
    \ob(\Pi_q) \ge \left \{ \begin{array}{ll}
    \frac{q^2+q-1-\frac{q^2}{p}}{3} & \mathrm{if~} p=3k+1, \\
    \frac{q^2+q+1-\frac{2q^2}{p}}{3} & \mathrm{if~} p=3k+2,~h~\mathrm{odd}, \\
    \frac{q^2+q-1-\frac{2q^2}{p}}{3} & \mathrm{if~} p=3k+2,~h~\mathrm{even}.
    \end{array} \right .
\]
\end{theorem}

The fourth section is dedicated to a probabilistic argument which will
give us a general lower bound on the balanced upper chromatic number
of projective planes.

\begin{theorem}\label{arbitrary} Let $\Pi_q$ be an arbitrary projective plane of order $q> 133$. Then its balanced upper chromatic number can be  bounded from below as $$ \overline{\chi}_b(\Pi_q)\geq \frac{q^2+q-16}{10}.$$
\end{theorem}


\section{Difference sets containing $\{0,1,3\}$}

We recall the definition of a difference set.

\begin{definition}
Let $G$ be a group of order $v$. A
$(v,k,\lambda)$\emph{-difference set} is a subset $D \subset G$ of size $k$
such that every nonidentity (nonzero) element of $G$ can be expressed as
$d_1d_2^{-1}$ (or $d_1-d_2$ if we use additive notation)
of elements $d_1, d_2 \in D$ in exactly $\lambda$ ways.
\end{definition}

Singer \cite{Singer} proved $\PG(2,q)$ admits a regular cyclic collineation
 group and thus can be
represented by a $(q^2+q+1,q+1,1)$-difference set in a cyclic (hence
abelian) group. For more details, see \cite{BJL}.

We start with a proof \cite{KissGy} of the fact mentioned above that
$
    \ob(\Pi_q) = \lfloor \frac{q^2+q+1}{3} \rfloor
$
if $\Pi_q$ comes from a difference set containing $\{0,1,3\}$.

\begin{proof} Let every $C_i$ in the partition consist of
three consecutive integers, with the possible exception of $C_1$ having
$4$ (this happens if $q=0,2$ mod $3$).
It is clear that this coloring is balanced with the above number of colors.
To show that no line is rainbow we note that every line contains a (unique)
 triple $\{j, j+1, j+3\}$. This triple is contained in the union of two
 consecutive  $C_i$'s, so by the pigeonhole principle two of them have the
 same color.
\end{proof}

We obtain a planar difference set by starting with a primitive cubic
polynomial $p(x)=x^3-ax^2-bx-c$ over $\GF(q)$ and now define the field
$\GF(q^3)=\GF(q)[x]/(p(x))$. Every monomial $x^i$ now reduces to a degree
(at most) 2 polynomial $c_2x^2+c_1x+c_0\equiv (c_2,c_1,c_0)\in\GF(q)^3$.
The exponents $i$, with $0\le i\le q^2+q$ for which $x^i$ lies in a
two-dimensional subspace now give a difference set. If we take the subspace
$c_2=0$, and if $a=0$, so $p(x)$ is of the form $x^3-bx-c$, then our
difference set will contain $0,1$ and $3$.

By a result S.D. Cohen \cite{Cohen} we know that a primitive polynomial
with this property exists for all $q\ne 4$. As a consequence, we get

\begin{proposition}
\[
    \ob(\PG(2,q)) = \left \lfloor \frac{q^2+q+1}{3} \right \rfloor.
\]
\end{proposition}

Note that the case $q=4$ has already been covered in \cite{KissGy}.

\section{Improving the lower bound on $\ob(\Pi_q)$ for certain classes
         of non-desarguesian planes}

We recall the proof of Theorem 2.3. in \cite{KissGy}.
For $0 \le i \le \frac{q^2+q+1}{3} -1$ define
the color classes as $C_i = \left \{ i, i+\frac{q^2+q+1}{3},
i+\frac{2(q^2+q+1)}{3} \right \}$. Since each line contains a
(unique) pair of points with difference $\frac{q^2+q+1}{3}$,
having therefor the same color, there are no rainbow
lines. Together with Result
\ref{AKMupper} this gives $\ob (\Pi_q) = \frac{q^2+q+1}{3}$, if
$q\equiv 1 \pmod{3}$.

In the following subsections we are going to investigate other
representations and improve some of the bounds.

\subsection{Using affine difference sets if $q\equiv 2 \pmod{3}$}
Our aim in this case is to use
affine difference sets and the corresponding
representation of affine planes and then add the ideal points to the
construction and color them in a suitable way.

\begin{definition}
Let $G$ be a group of order $q^2-1$, and let $N$ be a normal subgroup of
order $q-1$ of $G$. A $q$-subset $D$ of $G$ is called an
\emph{affine difference set} of order $q$ if $\{ d_1 d_2^{-1}:
d_1\ne d_2 \in D\} = G \setminus N$.
\end{definition}

An affine difference set $D$ gives rise to an affine plane (and hence to a
projective plane $\Pi_q(D)$) as follows: Points of the plane are the
elements of $G$, together with a special point $O$ (the origin),
lines through $O$ are the cosets of $N$, the remaining lines are of the
form $Dg$, $g\in G$.
We refer the reader to \cite{Bose}, \cite{Jung} and \cite{Hof} for further
 information about affine difference sets.

\begin{proof}[Proof of Theorem \ref{aff}]
We are going to define a coloring of the points in the orbit of size $q^2-1$,
and then give a suitable
coloring of the origin and the ideal points of the projective closure.
Similarly as above define the color classes as the right cosets of
a subgroup $T=\{1,t,t^2\}$ for a fixed element $t$ of order three,
so the color classes are of the form $C_g = \{ g, tg, t^2g \}$.
Note that $|N|$ is not divisible by $3$, so $t\not\in N$.
Since every element of $G\setminus N$ in particular $t$ is of the
form $d_1d_2^{-1}$ exactly once, this means that there will be two points
with the same color in every line which avoids the origin.

There are two things left to do: the first one is to color the origin and the
points of the ideal line in order to make sure that neither the lines through
the origin nor the ideal line are monochromatic.
This can be done in a greedy way. The origin, together with
three ideal points get a new color, the remaining $q-2$ ideal points $P$
get the color of one of the points on the line $OP$ in such a way that
no color is used twice (so altogether at most four times).

Observe that we now  indeed get a balanced coloring of the projective
plane and there are $\frac{q^2-1}{3}+1 = \frac{q^2+2}{3}$ color classes
such that exactly $q-1$ of them have 4 elements, the others have only 3
and there are no monochromatic lines.
\end{proof}

The value of the above result is questionable, since all known examples
of such planes are desarguesian.

\subsection{Using planar functions if $q\equiv 0 \pmod{3}$}

If $q=3^h$ for some $h\ge 1$ then we will use a representation of a
projective plane $\Pi_q(f)$ based on planar functions.

\begin{definition}
A function $f:\GF (q)\rightarrow \GF (q)$ is a {\em planar function} if the
equation $f(x+a)-f(x)=b$ has a unique solution in $x$ for every $a\not= 0$
 and every $b\in \GF (q)$.
\end{definition}

A planar function gives rise to an affine plane, and hence a projective plane
as follows. The point set will be the same as in
$\AG(2,q)$, the vertical lines with their ideal point remain the same but we
will replace every non-vertical line with a translate of the graph of $f$.
Note that parallel lines correspond to translations of $f$ that differ by
a vertical translation, and the ideal point of these translates can be
defined according to this.

We will
also assume that $f(0)=f(1)=0$. Moreover, let $H = \{-1,0,1\}$ be a
3-element subgroup in $(\GF(q), +)$, $q\equiv 0 \pmod{3}$. Thus the cosets of
the subgroup generated by $H$ and the vertical line through the origin
give us a partition of the point set into $\frac{q}{3}$ stripes.

\begin{proof}[Proof of Theorem \ref{par}]
The main idea in our coloring is to color in each of these stripes the 3
points which have the same second coordinate with the same color but do it in
such a way that the points with different second coordinates must have
pairwise different colors. Let us choose a representative system of the
cosets: $\{x_0,x_1,\ldots,x_{\frac{q}{3}-1}\}$. Let us denote the stripe
which contains $x_j$ with $S_j$. Therefore $S_j$ will correspond to the
points \mbox{$\{(x_j-1,y), (x_j,y), (x_j+1,y)\}$} for  $y\in \GF (q)$.

To begin with let us color the points $\{(x_0-1,0),(x_0,0),(x_0+1,0)\}$ from
$S_0$ and the ideal point of the vertical lines with the same color (let us
call it $C_{00}$). We can continue the coloring in $S_0$ by coloring the
triples $\{(x_0-1,y),(x_0,y),(x_0+1,y)\}$ with color $C_{0y}$. Similarly for
any $j=1,\ldots,\frac{q}{3}-1$ color the points $\{(x_j-1,y), (x_j,y),
(x_j+1,y)\}$ from $S_j$ with color $C_{jy}$. Notice that for any $j$ and $y$
the color $C_{jy}$ must be pairwise different.

Although for this coloring there will be no rainbow translates of $f$,
almost all of the vertical lines are rainbow. That is the reason why we
will modify this coloring a little bit. For every $j=1,\ldots,\frac{q}{3}-1$
delete $C_{j0}$ and $C_{j1}$. For every $j=0,\ldots,\frac{q}{3}-1$ denote
the ideal point of the translate of $f$ which goes through $(x_j-1,0)$ and
$(x_j,0)$ with $P_{j1}$; through $(x_j,0)$ and $(x_j+1,0)$ with $P_{j2}$;
through $(x_j-1,0)$ and $(x_j+1,0)$ with $P_{j3}$. Introduce new color
classes for every $j=1,\ldots,\frac{q}{3}-1$:
\[
    C_{j\alpha} = \{(x_j-1,0),(x_j-1,1),P_{j1}\}
\]
\[
    C_{j\beta} = \{(x_j,0),(x_j,1),P_{j2}\}
\]
\[
    C_{j\gamma} = \{(x_j+1,0),(x_j+1,1),P_{j3}\}.
\]

With this modification we certainly achieved that now every vertical line has
two points with the same color. Moreover, by coloring the appropriate ideal
points with these new colors we achieved that there are no rainbow
parabolas. But on the ideal line all of the points have pairwise different
colors so far. Notice that we did not color $P_{01}$, $P_{02}$ and $P_{03}$
yet. If we color these 3 points with a new color then this will take care of
the ideal line, too. One can see that there is only one color class $C_{00}$
which has 4 elements hence we used exactly $\frac{q^2+q}{3}$ color classes.
\end{proof}

\begin{figure}[!h]
    \centering
    \includegraphics[width=13cm]{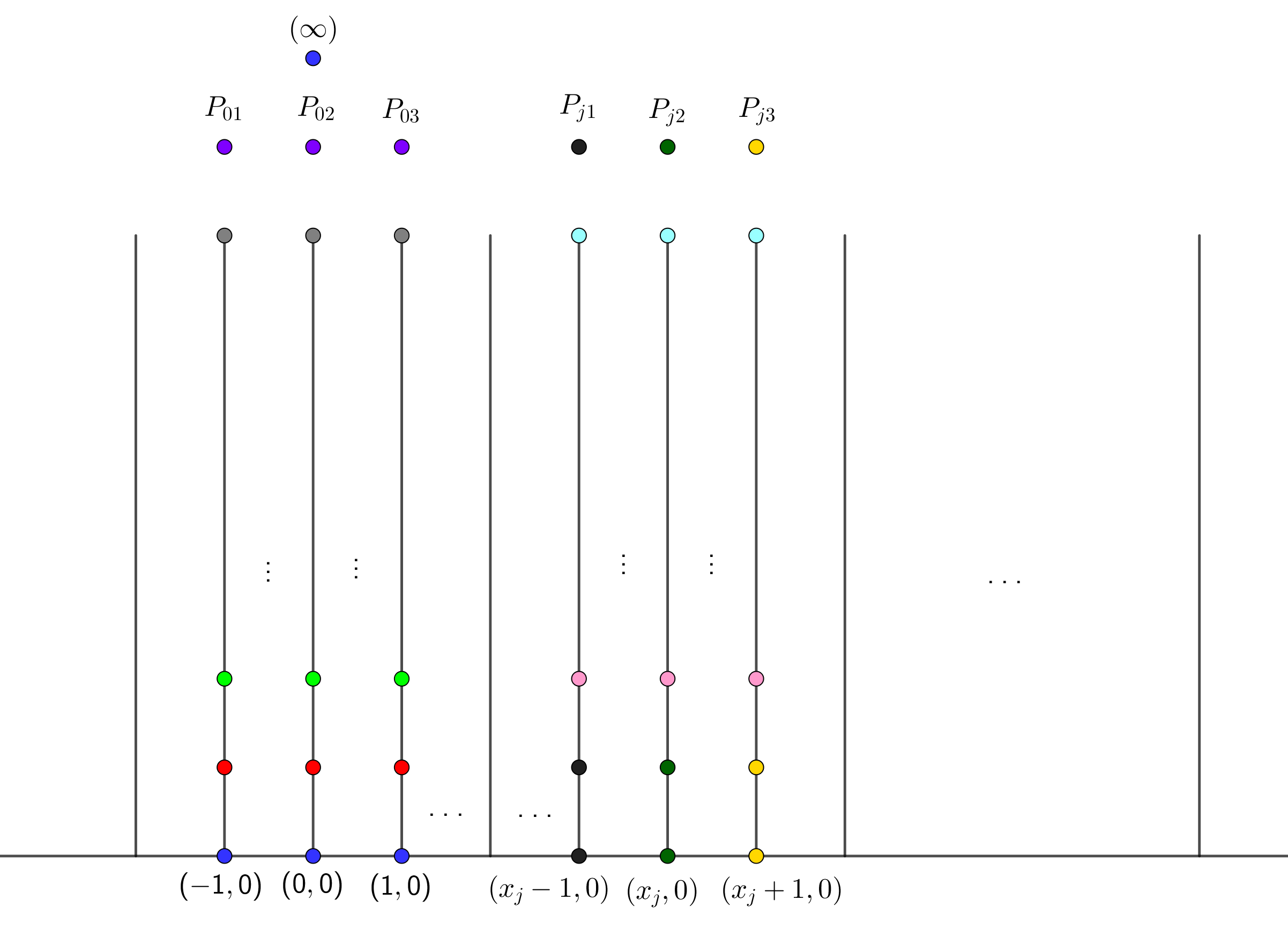}
    \caption{Modified coloring if $q \equiv 0 \pmod{3}$}
    \label{fig:my_label}
\end{figure}

\begin{remark}
In \cite{CoMat,DO} {Dembowski, Ostrom, Coulter, Matthews} showed that there
are planar functions so that the represented geometry
is not desarguesian.
\end{remark}

\newpage
\subsection{Using planar functions if $q\not \equiv 0 \pmod{3}$}

We discuss here two further constructions which give a bit weaker results for the desarguesian
projective plane but it completes the constructions for any projective plane which can be
represented with planar functions (which is a strictly larger class). Moreover, these constructions
give us a lower bound on the balanced upper chromatic number of the affine planes represented by
planar functions, too.

\begin{proof}[Proof of Theorem \ref{planar}]
Let $f$ be a planar function with $f(0)=f(1)=0$. Without loss of generality we can assume that $f(2)=1$ (otherwise
we can divide every value of $f$ with the value of $f(2)$). If $p=3k+1$ is a prime, $q=p^h$, then we will color
the affine plane of order $p$ represented by $f$. Every color class will be on two consecutive horizontal lines $y=c$ and $y=c+1$.

Point $(0,0)$ has color $1$, points $(1,0)$ and $(2,0)$ have color $2$, point $(3,0)$ has color $3$ and this pattern is repeated
until $(3k-3,0)$ (which is a single point). The last color class has 3 consecutive points, $(3k-2,0),(3k-1,0),(3k,0)$. On the line
$y=1$, the same pattern appears but everything is shifted by $x\to x+3$ so that the single color classes have pairwise different new
colors, and the pairs with the same color inherit their color from the single element of the previous line. More precisely, the points
$(1,1)$ and $(2,1)$ get color $1$, then point $(3,1)$ gets a new color, then $(4,1),(4,2)$ get the color of $(3,0)$, etc. At the end,
$(3k,1)$ gets a new color, and finally $(0,1)$ gets color $1$, too. With this coloring, on each horizontal line there will be three
consecutive points having the same color (these points are in a color class of size 4), the remaining color classes will have size 3.
After $p$ steps, we get back to the coloring on line $y=0$. Notice that this coloring also make the vertical lines rainbow-free.

If $p=3k+2$ and $p>5$ is a prime, $q=p^h$, then the pattern changes a little bit. We need to finish sooner the alternating sequence of single and
double classes, namely at $(3k-6,0)$, and then close with two 3 element classes separated with a single element with a new color. We
include the examples for $p=7$ and 11 in Figure \ref{peldak}.

\begin{figure}[!h]
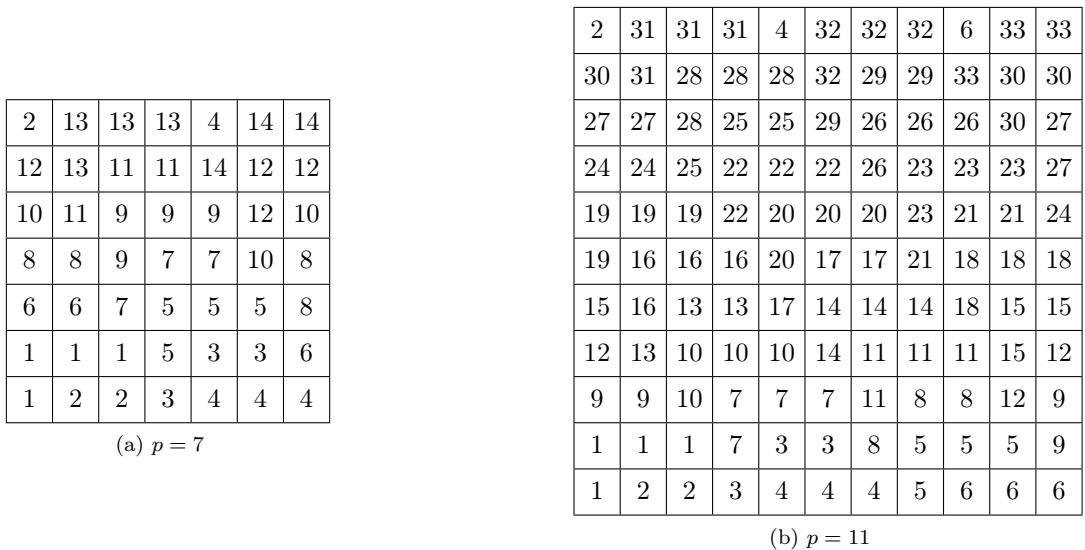

  \centering
  \subfloat[$p=7$ \label{7pelda}]{\setlength\tabcolsep{0pt}
\begin{tabular}{|@{\rule[-0.2cm]{0pt}{0.6cm}}*{7}{M{0.6cm} |}}
\hline
2 & 13 & 13 & 13 & 4 & 14 & 14 \\ \hline
12 & 13 & 11 & 11 & 14 & 12 & 12 \\ \hline
10 & 11 & 9 & 9 & 9 & 12 & 10 \\ \hline
8 & 8 & 9 & 7 & 7 & 10 & 8 \\ \hline
6 & 6 & 7 & 5 & 5 & 5 & 8 \\ \hline
1 & 1 & 1 & 5 & 3 & 3 & 6 \\ \hline
1 & 2 & 2 & 3 & 4 & 4 & 4 \\ \hline
\end{tabular}}
  \hspace{3cm}
  \subfloat[$p=11$ \label{11pelda}]{\setlength\tabcolsep{0pt}
\begin{tabular}{|@{\rule[-0.2cm]{0pt}{0.6cm}}*{11}{M{0.6cm} |}}
\hline
2 & 31 & 31 & 31 & 4 & 32 & 32 & 32 & 6 & 33 & 33 \\ \hline
30 & 31 & 28 & 28 & 28 & 32 & 29 & 29 & 33 & 30 & 30 \\ \hline
27 & 27 & 28 & 25 & 25 & 29 & 26 & 26 & 26 & 30 & 27 \\ \hline
24 & 24 & 25 & 22 & 22 & 22 & 26 & 23 & 23 & 23 & 27 \\ \hline
19 & 19 & 19 & 22 & 20 & 20 & 20 & 23 & 21 & 21 & 24 \\ \hline
19 & 16 & 16 & 16 & 20 & 17 & 17 & 21 & 18 & 18 & 18 \\ \hline
15 & 16 & 13 & 13 & 17 & 14 & 14 & 14 & 18 & 15 & 15 \\ \hline
12 & 13 & 10 & 10 & 10 & 14 & 11 & 11 & 11 & 15 & 12 \\ \hline
9 & 9 & 10 & 7 & 7 & 7 & 11 & 8 & 8 & 12 & 9 \\ \hline
1 & 1 & 1 & 7 & 3 & 3 & 8 & 5 & 5 & 5 & 9 \\ \hline
1 & 2 & 2 & 3 & 4 & 4 & 4 & 5 & 6 & 6 & 6 \\ \hline
\end{tabular}}
  \caption{Colorings of affine planes represented by planar functions}
  \label{peldak}
\end{figure}

If $h>1$ then we can extend these colorings of the similar $p\times p$ grids and get a balanced coloring of the affine plane defined by the planar
function $f$ in both cases.

These colorings use $p$ and $2p$ 4-element classes in every $p\times p$ grid, respectively. In total there are $\frac{q^2}{p^2}$ such grids which
means that in the affine plane of order $q$ the number of the 4-element color classes are

\[
\frac{q^2}{p}, \mathrm{~if~} p=3k+1 \qquad\qquad \mathrm{~and~} \qquad\qquad \frac{2q^2}{p}, \mathrm{~if~} p=3k+2.
\]

We can use these constructions in order to give a balanced coloring of the projective plane, too. Since the affine lines are rainbow-free, we can
arbitrarily color the points of the ideal line with 3 and 4 element color classes. Obviously we use the most colors if we use as many 3 element
color classes on the ideal line as we can. If $p=3k+1$ then $q+1\equiv 2 \pmod 3$ therefore in the ideal line there must be at least two 4-element
color classes. If $p=3k+2$ then the remainder of $q+1$ when divided by 3 depends on the parity of $h$. If $h$ is odd then $q+1\equiv 0 \pmod 3$, and if
even then $q+1 \equiv 2 \pmod 3$. In the following table we calculated the number of color classes of size 3 and 4 in every possible setup.

\begin{table}[!h]
\begin{tabular}{c|c|c|c|c|c|c|}
\cline{2-7}
\multirow{2}{*}{}                                                                          & \multicolumn{2}{c|}{$p=3k+1$}                                                                                                & \multicolumn{2}{c|}{$p=3k+2$, $h$ odd}                                                                                      & \multicolumn{2}{c|}{$p=3k+2$, $h$ even}                                                                                     \\ \cline{2-7}
                                                                                           & \begin{tabular}[c]{@{}c@{}}on the\\ ideal line\end{tabular} & \begin{tabular}[c]{@{}c@{}}in the \\ affine plane\end{tabular} & \begin{tabular}[c]{@{}c@{}}on the\\ ideal line\end{tabular} & \begin{tabular}[c]{@{}c@{}}in the\\ affine plane\end{tabular} & \begin{tabular}[c]{@{}c@{}}on the\\ ideal line\end{tabular} & \begin{tabular}[c]{@{}c@{}}in the\\ affine plane\end{tabular} \\ \hline
\multicolumn{1}{|c|}{\begin{tabular}[c]{@{}c@{}}the number of \\3-element color classes\end{tabular}} & $\frac{q-7}{3}$                                             & $\frac{q^2}{3}\cdot \frac{p-4}{p}$                             & $\frac{q+1}{3}$                                             & $\frac{q^2}{3}\cdot \frac{p-8}{p}$                            & $\frac{q-7}{3}$                                             & $\frac{q^2}{3}\cdot \frac{p-8}{p}$                            \\ \hline
\multicolumn{1}{|c|}{\begin{tabular}[c]{@{}c@{}}the number of \\4-element color classes\end{tabular}} & 2                                                           & $\frac{q^2}{p}$                                                & 0                                                           & $\frac{2q^2}{p}$                                              & 2                                                           & $\frac{2q^2}{p}$                                               \\ \hline
\end{tabular}
\end{table}
\end{proof}


The construction for $p=3k+2$ does not work for $p<11$. What can we say about $p=5$? Surprisingly, it turned out that for $p=5$ there is no balanced coloring of the affine plane of order 5 represented by a planar function $f$ with color classes of size 3 and 4. Moreover, none exists if there is at least one color class of size 4. These claims can be shown by a rather long case analysis which we choose to omit. However, by a computer search we found out that there exist a coloring such that all but one vertical line and one other line are rainbow, but we couldn't correct these errors by coloring the ideal points in order to get a balanced coloring of the projective plane of order 5 represented by a planar function $f$. It is straightforward to find a balanced coloring of the affine plane of order 5 with color classes of size 5 which can be generalized to get a balanced coloring for any affine and projective plane of order $q=5^h$ with roughly $\frac{q^2}{5}$ color classes. In Figure \ref{p5}, we included the above mentioned colorings of the affine plane of order 5.

\begin{figure}[!h]
  \centering
  \subfloat[,,almost'' good coloring for $p=5$  \label{5stefko}]{\setlength\tabcolsep{0pt}
\begin{tabular}{|@{\rule[-0.2cm]{0pt}{0.6cm}}*{5}{M{0.6cm} |}}
\hline
6 & 2 & 7 & \textcolor{red}{7} & \textcolor{red}{6} \\ \hline
4 & 7 & 5 & 6 & 4 \\ \hline
5 & 3 & 3 & 4 & 5 \\ \hline
3 & 1 & 1 & 5 & 1 \\ \hline
\textcolor{red}{1} & 2 & 2 & 2 & 3 \\ \hline
\end{tabular}}
  \hspace{3cm}
  \subfloat[with 5 element color classes \label{5with5}]{\setlength\tabcolsep{0pt}
\begin{tabular}{|@{\rule[-0.2cm]{0pt}{0.6cm}}*{5}{M{0.6cm} |}}
\hline
4 & 5 & 4 & 4 & 4 \\ \hline
3 & 3 & 4 & 3 & 3 \\ \hline
2 & 2 & 2 & 3 & 2 \\ \hline
1 & 1 & 1 & 1 & 2 \\ \hline
1 & 5 & 5 & 5 & 5 \\ \hline
\end{tabular}}
  \caption{Colorings of affine planes of order 5 represented by $f$}
  \label{p5}
\end{figure}


\section{General lower bound with a probabilistic approach}

In this section, we prove a general lower bound for all projective planes. In order to prove Theorem \ref{arbitrary}, we need a technical lemma which appeared in the paper of Nagy (\cite{ZLNagy}, Lemma 3.4).

\begin{lemma}\label{anal}  Denote $\prod_{i=1}^{k} \left(1-\frac{i}{n} \right)$ by $A_n(k)$. Then
	$$A_n(k)< \exp\left(-\left(\frac{k(k+2)}{2n-k-2}\right)\right)\Delta(n,k),$$
	
	where $\Delta(n,k)$ is the product of error terms $ \frac{\sqrt{n-1}}{\sqrt{n-k-1}}$, $ \left(1+  \frac{k^2}{12(n-k-1)^2}    \right)^k$ and $\left(1-\frac{(k+2)^2}{12n^2} \right)^{\frac{k(k+2)}{2n-k-2}}.$
\end{lemma}

We continue the preparation with a well known bound on the number of incidences $I(\mathcal{P}, \mathcal{L})$ between a point set $\mathcal{P}$ and a line set $\mathcal{L}$

\begin{lemma} [Incidence bound, see \cite{Bourgain}]\label{incidence}
	
	$$ I(\mathcal{P}, \mathcal{L})\leq \min \left\{ |\mathcal{P}|\sqrt{|\mathcal{L}|}+|\mathcal{L}| ,    |\mathcal{L}|\sqrt{|\mathcal{P}|}+  |\mathcal{P}|\right\}.$$
\end{lemma}

Note that this also follows from the upper bound of the Zarankiewicz number $z(m,n,2,2)$ \cite{Reiman}, which denotes then maximum number of $1$s in an $m\times n$ matrix which does not contain an all-$1$ $2\times 2$ submatrix.

\begin{proof}[Proof of Theorem \ref{arbitrary}]
	
	We show the existence of a suitable coloring with color classes of size $10$ or $11$ by the combination of  a probabilistic argument and an application of the point-line incidence bound Lemma \ref{incidence} together with Hall's marriage theorem. 
		
{\it (Step 1.)} Take an arbitrary point $Q$ of the plane and $t=\lceil cq \rceil$ lines incident to $Q$, where  the parameter $c\in (0,1)$ is determined later on. We choose uniformly at random a pair of points from each $\ell \setminus Q$ of these lines $\ell$ incident to $Q$, and we assign a distinct color to each pair.
	
{\it (Step 2.)} Write $tq=9s+r$, where $r\in \{0, 1, \ldots, 8\}$. Next we take a random coloring of the non-colored points of the set $\bigcup_{i=1}^t \ell_i\setminus Q$ so that apart from $r$ color classes of size $10$,  each color is used  $9$  times. We say that a color resolves a line if the line contains at least two points from that color class. And let's call a line resolved if it contains two points from the same color class.

	The probability that a line not incident to $Q$ is not resolved by this random coloring is less than
	
	$$1\cdot \left(1-\frac{8}{tq-1}\right)\cdot \left(1-\frac{2\cdot 8}{tq-2}\right)\cdots \cdot \left(1-\frac{(t-1)\cdot8}{tq-(t-1)}\right)<A_{tq/8}(t-1).$$
	
	Hence we may apply Lemma \ref{anal} to obtain that expected number of not resolved lines which are not incident to $Q$ is at most $$\mathbb{E}(\mbox{not resolved lines not on Q}) < q^2 \cdot A_{tq/8}(t-1)<q^2\cdot \exp\left( - \frac{t^2-1}{2tq/8-t-1}  \right)\Delta(tq/8, t-1),$$
	by the linearity of expectation.
	
	Here the right hand side can be bounded from above as
$$	q^2\cdot \exp\left( - \frac{t^2-1}{tq/4-t-1}  \right)\Delta(tq/8, t-1)< q^2\exp\left( - \frac{4t}{q}  \right)\exp\left(\frac{-t^2-t+q/4}{q/4(tq/4-t-1)}\right)\Delta(tq/8, t-1).$$

If $q>133$ holds and $t=\lceil cq \rceil$ is chosen appropriately, a careful calculation of the Taylor series of the error terms proves that the expected value can be bounded above by the main term  \begin{equation}\label{eq:0}\mathbb{E}(\mbox{not resolved lines not on Q}) <q^2\exp\left( - \frac{4t}{q}  \right)<q^2\exp(- {4c}).\end{equation}

To perform Step 3,	let us take a coloring as above with less than $q^2\exp(- 4c)$ lines not resolved, beside the lines through $Q$. In order to resolve these lines as well, our aim is to assign distinct not-colored points $P_f$ to each of these lines $f$ with $P_f\in f$, and choose a color for each assigned point from the colors used already on $f$. In this step we also require that every color must be used at most once. Finally in Step 4, we have to color the remaining uncolored points in such a way that all the lines through $Q$ are resolved and every color class is of size $10$ or $11$. Here we might apply new colors as well.
	
	
{\it (Step 3.)} First, we have to find a matching between the uncolored $(q-t+1)q+1$ points and the set of not resolved lines which are not incident to $Q$, that covers the set of the lines in view. To resolve at the end all the remaining lines as well (i.e., those that passes through $Q$), we extend this incidence graph by adding two copies of not resolved lines through $Q$ and joining them to the points incident to them. We apply  Hall's theorem twice combined with Lemma \ref{incidence} to prove the existence of the covering matching in view. The incident points chosen in this step to the lines are called the assigned points.

 Suppose that we have a set of not resolved lines $X$, and   a set $Y$ of uncolored points  of cardinality less than $|X|$ incident to them.  Lemma \ref{incidence} implies that $I(X,Y)< |X|\sqrt{|X|-1}+~|X|$ on the one hand, and we also know that $I(X,Y)\geq(q-t+1)|X|$. This is in turn a contradiction if
\begin{equation}\label{eq:1}
\sqrt{|X|-1}<(q-t), \mbox{ \ \ thus \ if \ } q^2\cdot\exp(-4c)+2(1-c)q < q^2(1-c)^2-2(1-c)q,
\end{equation}
where we took into account that $X$ is of size at most $q^2\exp(-4c)+2(1-c)q$, and the error term which may occur while we are considering the ceiling in  $t=\lceil cq \rceil$.

In other words, Condition \eqref{eq:1} yields a suitable  assignment of distinct not colored points  for the not resolved lines. To assign distinct colors for these points from their respective lines skew to $Q$, we apply Hall's theorem again and suppose to the contrary that there is a set $X$ of not resolved lines skew to $Q$ on which  less than $|X|$ colors were used. Hence the total number of colored points on these lines is at most $10|X|$ as each color can appear at most $10$ times. However, the number of incidences between the colored points of these lines could not exceed $|X|\sqrt{10|X|}+ 10|X|$ according to Lemma \ref{incidence}, while this incidence number is $t|X|=~ \lceil cq \rceil|X|$. Thus the got a contradiction to our assumption if
\begin{equation}\label{eq:2}
 |X|\sqrt{10|X|}+ 10|X|< \lceil cq \rceil|X| \mbox{ \ \ i.e., \ if \ } q^2\cdot\exp(-4c)< \frac{( cq -10)^2}{10} .
\end{equation}

If both Condition \eqref{eq:1} and \eqref{eq:2} hold then we are able to resolve all the lines skew to $Q$.
In order to choose the optimal constant $c$, we may suppose that these upper bounds are close to each other (asymptotically) i.e. we choose $c$ such that the values of $q^2(1-c)^2$ and $\frac{c^2q^2}{10}$ are almost the same.

If we pick that constant $c$ to be $c=0.77$, it is easy to verify that both Condition \eqref{eq:1} and \eqref{eq:2} hold when $q>133$.

{\it (Step 4.)} We finish are proof by coloring the remaining uncolored points such that all the lines containing $Q$ are resolved and each color is used $10$ or $11$ times. To guarantee the resolving property, we introduce new colors and take $5$ pair of assigned points from each distinct $5$ lines of the $q-t+1$ not yet resolved ones. This makes further color classes of size $10$, with less than $10$ assigned points left uncolored. These leftout assigned points finally get yet a new color, and this color class is completed to have size $10$ by putting in arbitrary uncolored points.\\
Up to this point, we already resolved all the lines but we have color classes of size $9$, $10$ and possibly $11$ as well. To end up with a balanced coloring we try to complete the classes of size $9$ to have size $10$ by coloring the remaining uncolored points. This is doable since the number of uncolored points is at least $(q-t+1)q+1- q^2\cdot\exp(-4c)$ which is more than the number $N<\frac{1}{9}tq$ of  color classes of size $9$ at this point.

Thus we obtained a balanced coloring of almost all the points which resolves every line.
Finally, we partition the remaining uncolored points to color classes of size $10$ and put the remaining at most $9$ points into distinct formerly created color classes of size $10$.
This provides at most $17$ classes of size $11$ and further classes of size $10$, which completes the proof.
\end{proof}

\bigskip

\begin{remark}\label{probkisq}
For $11\le q \le 133$, we are able to verify by a computer aided search that there exists a balanced coloring for an arbitrary projective plane of order $q$ with color classes of size 11 and 12, namely the number of colors needed is at most $ \frac{q^2 +q- 18}{11}$. One should repeat the steps in the proof of Theorem \ref{arbitrary} but use the concrete expected value instead of Condition (1) and use the stronger inequalities in Condition (2) and (3). Finally, suppose that the number of color classes is at most $ \frac{q^2 +q- 18}{11} $  and $q\leq 10$. Then we have less than $11$ colors, thus every line has a pair of monochromatic points by the pigeon-hole principle.
\end{remark}

\section*{Concluding remarks}
We showed that for certain non-desarguesian planes one can construct a rainbow-free coloring with color classes of size 3 and 4. It would definitely be interesting to find more classes of projective planes with this property. Some nice construction for potential planes can be found in the paper by Pott \cite{Pott}.

Probably, the most natural extension to our problem is to consider higher dimensional projective spaces. In this case, we can consider subspaces of fixed dimension $k$ in PG$(n,q)$, and try to determine the
balanced upper chromatic number. Some initial results in this direction can be found in Araujo-Pardo, Kiss, Montejano \cite{KissGy}.

One can also extend the problem to the case when more color class sizes are allowed. For example, one can consider rainbow-free coloring with color classes of size at most $k$, and determine the maximum number of
colors under this condition.

\section*{Acknowledgement}
\protect\includegraphics[height=1cm]{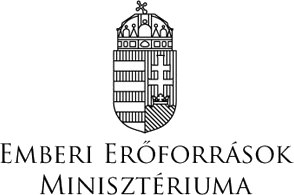} The first author was supported by the \'UNKP-18-3 New National Excellence Program of the Ministry of Human Capacities. In the first part of this research, the first, third and fifth authors gratefully acknowledge the support of the bilateral Slovenian--Hungarian Joint Research Project no.\ NN 114614 (in Hungary) and N1-0140 (in Slovenia). In the second part of this research, these three authors were supported by the Slovenian--Hungarian Bilateral project Graph colouring and finite geometry (NKM-95/2019000206) of the two Academies. The third author was supported by research project N1-0140 of the Slovenian Research Agency. The fourth author is also supported by the Hungarian Research Grants (NKFI) No. K 120154 and SNN 132625 and by the J\'anos Bolyai Scholarship of the Hungarian Academy of Sciences. The fifth author was also supported by research project J1-9110 of ARRS.

\end{document}